\documentclass[11pt,twoside,a4paper,reqno]{amsart}%
\usepackage{amssymb} %\rightrightarrows
\usepackage[OT4]{fontenc}
\usepackage[cp1250]{inputenc}
\usepackage{pstricks}
\usepackage{pst-plot}
\usepackage{mathtools} %\cramped
\usepackage{lineno}
\usepackage[bookmarksopen = true]{hyperref}
\def\mr#1 (#2){\href{http://www.ams.org/mathscinet-getitem?mr=#1}{MR: #1 (#2)}}
\def\MR#1 (#2){\href{http://www.ams.org/mathscinet-getitem?mr=#1}{MR: #1 (#2)}}
\def\MRa#1 (#2){\href{http://www.ams.org/mathscinet-getitem?mr=#1}{MR: #1}
\href{http://www.ams.org/mathscinet-getitem?mr=#1}{(#2)}}
\newcommand{\N}{\ensuremath{\mathbb{N}}}

\newcommand{\R}{\ensuremath{\mathbb{R}}}
\newcommand{\T}{\ensuremath{\mathbb{T}}}
\newcommand{\Z}{\ensuremath{\mathbb{Z}}}
\newcommand{\e}{\ensuremath{\varepsilon}}
\newcommand{\f}{\ensuremath{\varphi}}
\newcommand{\m}{\ensuremath{\mu}}
\edef\od{\:}
\def\:{\colon}
\def\df{\mathrel{\mathop:}=}
\def\too{\longrightarrow}
\newcommand{\dm}[1]{\mathrm{d}#1}
    \newtheorem{thm}{Theorem}[section]
    \newtheorem{con}{Conclusion}[section]
\theoremstyle{remark}
    \newtheorem{rem}[thm]{Remark}
\newcommand{\n}[1]{}
\newcommand{\ko}[1]{%
                    }%

\title[Transitive cylinders with full-dimensional discrete points]%
	{Transitive cylinder flows\\
	 whose set of discrete points\\
	 is of full Hausdorff dimension}
\author[Eugeniusz Dymek]{Eugeniusz Dymek}
\newcommand{\qqk}[1][l]{\ensuremath{q_{1+k_{#1}}}}
\newcommand{\qk}[1][]{\qqk[n #1]}
\newcommand{\qq}[1][l]{\ensuremath{q_{k_{#1}}}}
\newcommand{\q}[1][]{\qq[n #1]}
\newcommand{\qqa}[1][l]{\ensuremath{A_{#1}q_{k_{#1}}}}
\newcommand{\qa}[1][]{\qqa[n #1]}

\newcommand{\A}[1][]{\ensuremath{A_{n #1}}}
\renewcommand{\L}[1][]{\ensuremath{L_{n #1}}}

\renewcommand{\t}[1][]{\ensuremath{T_{f_{#1}}}}
\newcommand{\xr}{\ensuremath{X\times\R}}

\begin{document}
%\linenumbers 
\begin{abstract}
For each irrational $\alpha\in[0,1)$ we construct a continuous function $f\: [0,1)\to \R$
such that the corresponding cylindrical transformation $[0,1)\times\R \ni (x,t)
\mapsto (x+\alpha, t+ f(x)) \in [0,1)\times\R$ is transitive and the Hausdorff dimension of the set
of points whose orbits are discrete is 2. Such cylindrical transformations are shown
to display a certain chaotic behaviour of Devaney-like type.
%
%A cylindrical transformation (cylinder) is a mapping of the form
%\(
%	T_f\:X\times\R\ni (x,t) \mapsto (Tx, t+f(x))
%\)
%where $T\: X\to X$ is a homeomorphism of a topological space, and is a
%continuous function. It may display varied dynamical behaviour. In particular, Fr\k
%aczek and Lema\'nczyk found for every minimal rotation $T$ of a torus $\T^d$ a function
%$f$ for which $T_f$ has both dense and closed discrete orbits (a Besicovitch cylinder).
%Also, for $d=1$ and almost every irrational rotation they constructed $f$ such that the
%Hausdorff dimension of the set of points with discrete orbit under $T_f$ is at
%least $3/2$. By enhancing their construction we show that for these rotations this set
%can actually have full dimension (i.e.\ $2$). We recall also an application to
%differential equations.
\end{abstract}

\maketitle

\section*{Introduction}
Chaotic behaviour in dynamical systems has been of particular interest in topological
dynamics since about the second half of the 20th century.%
	\footnote{Let us remind that one of the most popular definitions of chaos, the
	Devaney chaos, comprises dense orbits (transitivity), dense set of periodic orbits
	and sensitivity to initial conditions (the last condition usually follows from the
	first two ones).}
Few examples had been studied earlier, thus it must have been surprising to Abram
Besicovitch to discover a homeomorphism of the cylinder $\T\times\R$ ($\simeq [0,1)\times\R$)
with both dense and (closed) discrete orbits (\cite{bes2}, see also \cite{bes1}).
It is an example of a class called now cylindrical transformations or, more generally,
skew products. \emph{Cylindrical transformation} or \emph{cylinder}%
	\footnote{Also called \emph{cylinder flow} or \R-extension.}
is a mapping of the form
\[
	T_f\:X\times\R\ni (x,t) \mapsto (Tx, t+f(x)) \in X\times\R,
\]
where $T\: X\to X$ is, in the most general setting, a homeomorphism of a topological
space, and $f\: X\to \R$ is a continuous function. They arise naturally in ergodic
theory, as their iterates are the products of respective iterates of $T$ and the ergodic
sums of $f$ over $T$. Formally, they were introduced (and even earned
their own chapter) in a textbook on topological dynamics \cite[Chapter 14]{gohe}.
However, such transformations were already considered before: Besicovitch viewed his
cylinder on $[0,1)\times\R$ as a homeomorphism of the punctured plane $\R^2\setminus\{0\}$
and expanded it to a homeomorphism of the plane with dense and discrete orbits. Also, cylinders are sections
of the flows derived from some differential equations (studied in \cite{poi}, Chapitre XIX,
pp. 202ff.; see also \cite[Section 8]{flem}).

The result of Besicovitch concerned only some particular $Tx=x+\alpha$ on $\T$ and
$f\:\T\to\R$. Therefore, a few natural questions arise: which cylinders have both dense
and discrete orbits? (Such cylinders are hereafter called \emph{Besicovitch cylinders}.)
For which rotations do such cylinders exist and how common are they? How many discrete orbits do they have?
How about other homeomorphisms $(X,T)$? These
problems were studied, among others, by Fr\k aczek and Lema\'nczyk in \cite{flem} and
by Kwiatkowski and Siemaszko in \cite{ksie}. In particular, in \cite{flem},
Besicovitch cylinders over every minimal rotation of tori $\T^d$ were constructed.
As for the amount of discrete orbits, it is known that the set of nonrecurrent points
in these cases
is small in both topological and measure-theoretical sense: it is of first
category (albeit dense) and of measure zero. Thus, the authors of \cite{flem} used some finer means 
to analyse the set of points with discrete orbits. Firstly, for every minimal rotation
of a torus $\T^d$ they found a Besicovitch cylinder with uncountably many discrete orbits.
Secondly, for almost every minimal rotation there is a Besicovitch cylinder for which
the points with discrete orbits have altogether the Hausdorff dimension at least $d + 1/2$
(that is, of codimension at most $1/2$). Also, the authors discovered some classes of regular examples
(in terms of H\"older continuity, Fourier coefficients or degree of smoothness).
They left as an open problem whether higher Hausdorff dimensions can be achieved. The (positive)
solution this problem is the main topic of the present paper: by enhancing the techniques from \cite{flem}
we have constructed  Besicovitch cylinders with full Hausdorff dimension of discrete orbits for
every minimal rotation of $\T^d$.

The present paper consists of four sections. Section 1 contains some preliminary facts on
cylindrical transformations that are relevant to our quest for Besicovitch cylinders;
in particular, we show that they form a first category set within some relevant function space.
In Sections 2 and 3, we present our construction, define some subsets of
\T\ and prove that their elements have discrete orbits (although there may also exist
other discrete orbits). The Hausdorff dimension
of these sets is calculated in Section 4. The last section introduces
a definition of chaos that some Besicovitch cylinders satisfy, which is also a possible
generalization of the Devaney chaos to noncompact dynamical systems.

\section{Cylindrical transformations}
\label{ss/cyl}
The cylindrical transformations are a special case of the concept of skew product
(see \cite{foko}, Subsection 10.1.3) in ergodic theory, transferred in a~natural way to
the topological setting. In general, they can be defined for a minimal homeomorphism $T$
of a compact metric space $X$ (the \emph{base}) with a $T$-invariant measure \m\ defined on
the Borel
$\sigma$-algebra, and a real continuous function $f\:X\too \R$ (which we will
customarily call a \emph{cocycle}). In the next sections, we will confine
ourselves to minimal rotations on tori with Lebesgue measure.%
    \footnote{Minimal rotations on compact groups do not always exist -- groups
    possessing them are
    called \emph{monothetic}. All tori $\T^n$ are monothetic, and a rotation on
    $\T^n$ is minimal precisely when its coordinates are irrational and $\mathbb{Q}$-%
    linearly independent; moreover, these rotations are uniquely ergodic with respect to
    Lebesgue measure.}
Now, $T$ and $f$ generate
a \emph{cylindrical transformation} (or a \emph{cylinder}):
\begin{align*}
    &T_f\:X\times\R\too X\times\R\\
    &T_f(x,t) \df (Tx,t+f(x))
\end{align*}
The iterations of $T_f$ are of the form $T_f^n(x,t) = (T^n x,t+f^{(n)}(x))$,
where $f^{(n)}$ is given by the formula:
\[
    f^{(n)}(x)\df
    \begin{cases}
    \ f(x) + f(Tx) + \dots + f(T^{n-1}x),& \text{for } n>0,\\
    \quad 0, & \text{for } n=0,\\
    \ - f(T^{-1}x) - f(T^{-2}x) - \dots - f(T^n x),& \text{for } n<0.
    \end{cases}
\]
Observe that the dynamics of a point $(x,t)$ does not depend
on $t$, because the mappings
\[
    \tau_{t_0}\: X\times\R \ni (x,t)\mapsto (x,t+t_0) \in X\times\R
\]
(for arbitrary $t_0\in \R$) are in the topological centralizer of $T_f$:
\begin{multline*}
    T_f(\tau_{t_0}(x,t)) = T_f(x,t+t_0) = (Tx,t+t_0+f(x)) =\\
    \tau_{t_0}(Tx,t+f(x)) = \tau_{t_0}(T_f(x,t)).
\end{multline*}

Unlike the compact case, a homeomorphism of a locally compact space (as here $\xr$) need
not have a minimal subset. The cylinder on a compact space is never minimal (as proved in
\cite{bes2}), thus it is meaningful to study minimal subsets.

From now on, we will usually assume that the base is a torus $X=\T^d$ or even the circle (with a minimal
rotation). Then, it is well-known that there are two cases in which the minimal subsets
can be easily described:
\begin{enumerate}
\def\theenumi{T\arabic{enumi}}
    \item when $\int_{\T^d} f\dm\mu \neq 0$, all points
    have closed discrete orbits, or, equivalently: for all $x\in X$: $|f^{(n)}(x)|
    \xrightarrow{n\to\infty} \infty$. \label{tr1}
    \item when the cocycle is a \emph{coboundary}, i.e.\ of the form $f= g - g\circ
    T$ for a (continuous) \emph{transfer function} $g\:\T^d\too\R$, the minimal sets are
    vertically translated copies of the graph of $g$ in $\T^d\times \R$. Conversely, if some
    orbit under $T_f$ is bounded, then so are all of them, and the cocycle $f$ is a
    coboundary (Gottschalk-Hedlund Theorem, \cite[Theorem 14.11]{gohe}).\label{tr2}
\end{enumerate}
Notice that if $f$ is
a coboundary and $T$ is measure-preserving, then $\int_{\T^d} f\dm\mu = 0$. Also, in both
cases \ref{tr1} and \ref{tr2} the phase space decomposes into minimal sets. In what follows we
will call cocycles that fulfil \ref{tr1} or \ref{tr2} \emph{trivial}.

\begin{thm} [Lema\'nczyk, Mentzen] 
	\label{t1}
	If $X=\T^d$, $T$ is a minimal rotation and a cocycle
	$f$ is not trivial, then the cylinder $T_f$ is transitive (see \cite{leme}, Lemmas 5.2,
	5.3). Therefore, if $f$ is of average zero, but $T_f$ has a closed discrete orbit, then
	it is automatically transitive (since by \ref{tr2} coboundaries have only bounded orbits).
\end{thm}

As in \cite{flem}, we consider cylindrical transformations that display both transitive
and discrete behaviour, called \emph{Besicovitch transformations} or \emph{Besicovitch
cylinders}, because their first example was given in \cite{bes2}. Given a homeomorphism
of the base, we will also call a cocycle which generates a Besicovitch cylinder
a \emph{Besicovitch cocycle}. For brevity, we will also write `discrete' instead of
`closed discrete'. By virtue of the condition \ref{tr1} and Theorem \ref{t1},
\emph{a cocycle is Besicovitch if and only if it has average zero and the resulting
cylinder has a discrete orbit}. This characterisation will be used in our paper. 

Unfortunately, Besicovitch cylinders are not easy to find. When $T$ is an irrational
rotation of the circle, too regular cocycles yield no minimal sets at all, as has been
proved by Matsumoto and Shishikuro, and, independently, by Mentzen and Siemaszko:
\begin{thm}[{\cite[Theorem 1]{mash}, \cite[Theorem 2.4]{mesi}}]
\label{bv-min}
If the cocycle on \T\ is nontrivial and of bounded variation, then the cylinders which
it generates have no minimal sets. In particular, they have no discrete orbits.
\end{thm}
Moreover, the set of Besicovitch cocycles for any minimal compact base is
first category in the set of all cocycles with zero average -- for a proof,
see Subsection \ref{ss1}.

Given a cylinder $T_f$, we will denote
\begin{align*}
    \mathcal D\df&\; \{x\in X\mathpunct{:} \text{ the $T_f$-orbit of }(x,t) \text{ is discrete for every } t\in\R\}\\
    =&\; \{x\in X\mathpunct{:} \text{ the $T_f$-orbit of }(x,0) \text{ is discrete}\}.
\end{align*}
Then the set of points in \xr\ with discrete orbits equals $\mathcal D\times\R$.
Clearly, if $T$ is minimal and $\mathcal D\neq\emptyset$, then both $\mathcal D$
and $\mathcal D\times\R$ are dense in ambient spaces, as $\mathcal D$ is
$T$-invariant.

In \cite{flem}, the authors construct a Besicovitch cocycle for any minimal
rotation of a torus. They also find ones with some special properties,
in particular with relatively large $\mathcal D$.%
    \footnote{Recall also that the set of discrete orbits is of first category and of
    measure zero for minimal rotations of tori.}

\begin{thm}[\cite{flem}]
\label{fl1}
For every irrational rotation of \T\ there exist Besicovitch cocycles
\textup{(\cite[Section 2]{flem})}. The cocycles can be chosen in such a way that
$\mathcal D$ is uncountable \textup{(\cite[Proposition 6]{flem})}. Moreover, for almost
every irrational rotation one can find Besicovitch cocycles such that the Hausdorff
dimension of $\mathcal D$ is at least $1/2$ \textup{(\cite[Theorem~9]{flem})}.  
\end{thm}
It was left as an open problem whether the coefficient $1/2$ could be
improved or not. We answer it by developing the techniques from \cite{flem}:
for every irrational rotation of \T\ we have obtained Besicovitch cylinders with
$\mathcal D$ of full Hausdorff dimension (Conclusion \ref{con}). This construction is
presented in Section \ref{s2}.
\subsection{Nonrecurrent cylinders are first category}
\label{ss1}
We aim to show that, given a uniquely ergodic homeomorphism of a compact metric space as the base,
all cocycles admitting nonrecurrent orbits form a first category set in the space of zero-averaged
cocycles (with the uniform topology). In particular, we will prove that 
Besicovitch cocycles are of first category. Note that a minimal rotation of a compact metric group
is uniquely ergodic for the Haar measure.
\begin{proof}
Denote by $(X,\mu)$ the space, by $T$ a~uniquely ergodic homeomorphism thereof, and by $f\:X\too\R$
a cocycle. Also, $\lvert  p - q\rvert $ will denote the distance between $p,\,q\in\xr$
in the taxicab metric.

Recall that $p\in\xr$ is \emph{nonrecurrent} for \t\ if it is not recurrent, i.e. if its
positive semi-orbit lies outside some neighbourhood of $p$: there exists $\e>0$ such 
that $\lvert p - \t^k(p)\rvert\geq \e$ for every $k>0$. Thus, all the functions in question
are contained in the union (increasing as $\e\to 0$ or $n\to\infty$) $\bigcup_{\e>0} N_\e = \bigcup_{n=1}^\infty
N_{1/n}$, where
	\begin{align*}
	N_\e \df \{f\:X\to\R\mathpunct{:}\ &f \text{ is continuous, }\int_X f\,\mathrm{d}\mu = 0,\\
							&\lvert p-\t^k(p)\rvert\geq\e \text{ for some } p\in\xr
							\ \text{and all } k>0\}
	\end{align*}

To finish the proof, we will show that every $N_\e$ is closed and has empty interior --
hence their union, by definition, is of first category. From now on, an $\e>0$ will be
fixed.
\paragraph{\textnormal{1}} The set $N_\e$ has empty interior because the set of coboundaries
%\footnote{I.e. the functions of the form $g-g\circ T$ for some continuous $g\: X\too\R$}
is dense (by the ergodic theorem for uniquely ergodic homeomorphisms), and the cylinders generated
by coboundaries have only recurrent points, so all coboudaries lie outside $N_\e$.
\paragraph{\textnormal{2}} To prove that $N_\e$ is closed, consider a uniformly convergent sequence
$(f_j)_{j=1}^\infty\subset N_\e$, $f_j\rightrightarrows f$. Let $p_j\in\xr$ be chosen for $f_j$
as in the definition of $N_\e$. We may assume that all $p_j$ lie in $X\times\{0\}$,
because the dynamic behaviour of a point wrt \t\ does not depend on its second
coordinate. Since $X$ is compact, $(p_j)$ has an accumulation point, say, $p_{j_n}\too p$.
We will show that this point satisfies the condition from the definition of $N_\e$ for
$f$.
%\begin{gather*}
%	\e \leq d(p_{j_n}, \t[j_n]^k(p_{j_n})) \leq d(p_{j_n},p) + d(p, \t^k(p)) +
%	d(\t^k(p), \t^k(p_{j_n})) + d(\t^k(p_{j_n}), \t[j_n]^k(p_{j_n}))
%\end{gather*}

By the choice of $p_j$, the following holds for all $k>0$ and $n>0$:
\begin{align*}
	\e &\leq \lvert p_{j_n} - \t[j_n]^k(p_{j_n})\rvert\\
	&\leq \lvert p_{j_n} - p\rvert + \lvert p - \t^k(p)\rvert
	+ \lvert\t^k(p) - \t^k(p_{j_n}) \rvert + \lvert\t^k(p_{j_n}) - \t[j_n]^k(p_{j_n})
	\rvert.
\end{align*}
After passing to the limit as $n\to\infty$ all but the second of the summands vanish.
Indeed, this is obvious for the first and the third one. As for the last summand, it follows form
the convergence $f_j\rightrightarrows f$: one can easily check that the supremum distance 
between arbitrary $T_g^k$ and $T_{g'}^k$ equals the supremum distance $\|g^{(k)}
- g'^{(k)}\|_{\sup}$, which is at most $k\|g - g'\|_{\sup}$, so
$\lvert\t^k(p_{j_n}) - \t[j_n]^k(p_{j_n})\rvert \leq k\|f - f_{j_n}\|_{\sup} \to 0$.
This finally proves that $\lvert p - \t^k(p)\rvert\geq \e$ for all $k>0$, and therefore
$f\in N_\e$.
\end{proof}
\subsubsection*{Remark} The proof remains valid for each Banach subspace $\mathcal F
\subset \mathcal C (X)$, satisfying the ergodic theorem, whose norm is
stronger than $\|\cdot\|_{\sup}$ and on which $T$ acts as an isometry (in particular,
for the space of H\"older continuous functions and for $\mathcal{C}^k(\T^d) \subset
\mathcal{C}(\T^d)$).
\section{Construction of a Besicovitch cylinder}
\label{s2}
Let $\alpha$ be an irrational number in $[0,1)$ and $(p_n/q_n)_{n\geq 0}$ its sequence of
convergents. Recall that then
 \begin{equation}
	\frac 1{2q_n q_{n+1}} < (-1)^n \left(\alpha-\frac{p_n}{q_n}\right)
	= \left\lvert\alpha-\frac{p_n}{q_n}\right\rvert < \frac 1{q_n q_{n+1}};
	\label{e0}
 \end{equation}
(by \cite{khi}, Theorems 9 and 13). Because $q_n\to\infty$, one can choose a subsequence
$(\q)_{n\geq 1}$ that grows quickly enough:
 \begin{gather}
	\qqk[1] \geq 9, \label{e4}\\
 	\q[+1] \geq 5 \q, \ \qk[+1] \geq 5\qk, \label{e1}\\
 	\frac 1n \log\q \to \infty. \label{e3}
 \end{gather}
We may also assume that
 \begin{equation}
	\text{all } k_n \text{ are even or all are odd.}
	\label{e17}
 \end{equation}
For example, we can set $k_n\df 4 n^2 + 1$, because always $q_6 \geq \operatorname{Fib}_6 =13$,
$\q[+1] \geq q_{4 + k_n} \geq \operatorname{Fib}_4 \q = 5 \q$ and the sequences $q_{4n^2 + 1} \geq
\operatorname{Fib}_{4n^2 + 1}$ grow superexponentially. Put additionally
 \begin{equation}
	\A\df \lfloor (3/4)^n\qk \rfloor >  (3/4)^n\qk -1 \text{\quad for } n\geq 1.
	\label{e2}
 \end{equation}
%We need here the coefficients $\A$ to grow appropriately slower than $\qk$. The exact rate
%depends on the growth of... \textsl{A mo\.ze tak:} The above conditions are actually introduced
%to... \textsl{A mo\.ze po prostu zacz\k a\'c od po\.z\k adanych w\l asno\'sci \q\ i \A? A potem
%poda\'c konkretniejsze przyk\l ady. A w og\'ole to teraz si\k e tym nie przejmowa\'c...}
%
It follows that for $n\geq 2$ on the one hand
%% \begin{align}
%%	&\frac {\A\qk[-1]}{\A[-1]\qk} \stackrel{\eqref{e2}}>\frac {((3/4)^n\qk -1)\od\qk[-1]}
%%	 {(3/4)^{n-1}\qk[-1]\qk} = \frac 34 - \frac{(4/3)^{n-1}}{\qk}
%%	 \label{e5}\\
%%	&\phantom{\frac {\A\qk[-1]}{\A[-1]\qk} > ((3/4)^n 3)}	 
%%	 \stackrel{\eqref{e1}}\geq \frac 34 - \frac{(4/3)^{n-1}}{5^{n-2}\qqk[2]}
%%	 \geq \frac 34 - \frac {4/3}{\qqk[2]} \stackrel{\eqref{e4},\, \eqref{e1}} > \frac{18}{25},
%%	 \nonumber
%%	\intertext{and on the other hand}
%%	&\frac {\A[-1]\qk}{\A\qk[-1]} \stackrel{\eqref{e2}}>\frac {((3/4)^{n-1}\qk[-1]-1)\od \qk}
%%	 {(3/4)^n\qk\qk[-1]} = \frac 43 - \frac{(4/3)^n}{\qk[-1]} \\
%%	&\phantom{\frac {\A\qk[-1]}{\A[-1]\qk} > ((3/4)^n 3)}
%%	 \stackrel{\eqref{e1}}\geq \frac 43 - \frac{(4/3)^n}{5^{n-2}\qqk[1]}
%%	 \geq \frac 43 - \frac {16/9} {\qqk[1]} \stackrel{\eqref{e4}}> 1.1,
%%	 \nonumber
%% \end{align}
 \begin{multline}
	\frac {\A\qk[-1]}{\A[-1]\qk} \stackrel{\eqref{e2}}>\frac {((3/4)^n\qk -1)\od\qk[-1]}%
	 {(3/4)^{n-1}\qk[-1]\qk} = \frac 34 - \frac{(4/3)^{n-1}}{\qk}
	 \label{e5}\\	 
	 \stackrel{\eqref{e1}}\geq \frac 34 - \frac{(4/3)^{n-1}}{5^{n-2}\qqk[2]}
	 \geq \frac 34 - \frac {4/3}{\qqk[2]} \stackrel{\eqref{e4},\, \eqref{e1}} > \frac{18}{25},
 \end{multline}
and on the other hand
 \begin{multline}
	\frac {\A[-1]\qk}{\A\qk[-1]} \stackrel{\eqref{e2}}>\frac {((3/4)^{n-1}\qk[-1]-1)\od \qk}
	 {(3/4)^n\qk\qk[-1]} = \frac 43 - \frac{(4/3)^n}{\qk[-1]} \\
	 \stackrel{\eqref{e1}}\geq \frac 43 - \frac{(4/3)^n}{5^{n-2}\qqk[1]}
	 \geq \frac 43 - \frac {16/9} {\qqk[1]} \stackrel{\eqref{e4}}> 1.1,
 \end{multline}
hence altogether% \end{align}hence altogether\begin{align}%
 \begin{gather}
	 1.1 < \frac {\qk}{\A} : \frac{\qk[-1]}{\A[-1]} < 25/18.
	\label{e6}
\intertext{This also proves that}
	 \text{the sequence } \qk/\A \text{ rises exponentially.} \label{e11}
 \end{gather}
For the sake of brevity, we will also denote
 \[
	\L\df \q\qk/n^2, \quad \text{for } n\geq 1.
 \]
 
We consider a modification of the example from \cite[Section 2]{flem}: we define $f_n$ to be
\L-Lipschitz, $1/(\qa)$-pe\-ri\-od\-ic and even continuous function (hence also $f_n(\frac
1{\qa} - x) = f_n(x)$) by the formulas:
 \[
 	f_n(x)\df
	 \begin{cases}
	\phantom{.}\ 
	0, &\text{for \ } 0\leq x \leq \dfrac 1{12\qa},\\
	\phantom{\Bigg|}\ 
	\L\left(x- \dfrac 1{12\qa} \right), &\text{for \ } \dfrac 1{12\qa} \leq x
		\leq \dfrac 5{12\qa},\\
	\phantom{\Bigg|}\  
	\dfrac {\qk}{3\A n^2}, &\text{for \ } \dfrac 5{12\qa} \leq x \leq \dfrac 1{2\qa}.
	 \end{cases}
 \]
By periodicity:
 \begin{multline*}
	\lvert f_n(x+\alpha) - f_n(x)\rvert = \left\lvert f_n\left(x+\alpha - \frac{\A p_{k_n}}
	{\qa}\right) - f_n(x)\right\rvert \\
	\stackrel{f_n \text{ is $\L$-Lipsch.}}\leq \L \left\lvert \alpha - \frac{p_{k_n}}{\q}\right\rvert
	\stackrel{\eqref{e0}}< \frac {\q\qk}{n^2}
	\frac 1{\q\qk} = 1/n^2,
 \end{multline*}
so the series
 \begin{gather}
	\f(x) \df \sum_{l=1}^\infty (f_l(x+\alpha) - f_l(x))\nonumber
\intertext{converges uniformly and yields a continuous cocycle of average zero.
Moreover, it is easy to verify that for every $m\in\Z$:}
	\f^{(m)}(x) = \sum_{l=1}^\infty (f_l(x+m\alpha) - f_l(x)),\label{e8}
 \end{gather}
where $\f^{(m)}(x)$ is the second coordinate of $T_\f^m (x,0)$ (we recall that
$T_\f^m (x,t) = (T^m(x), t+ \f^{(m)}(x))$ for every $x$ and $t$).
\section{Discrete orbits}
Consider, for $n\geq 1$ and $j=0,\dotsc,\qa -1$:
 \begin{align*}
	&F^{++}_{n,j} \df \left[-\frac 1{12\qa} , \frac 1{12\qa}\right] + \frac j\qa,
	 \displaybreak[0]\\
	&F^{-+}_{n,j} \df \left[\frac 1{6\qa} , \frac 1{3\qa}\right] + \frac j\qa,\displaybreak[0]\\
	&F^{--}_{n,j} \df \left[\frac 5{12\qa} , \frac 7{12\qa}\right] + \frac j\qa = F^{++}_{n,j}
	 + \frac 1{2\qa},\displaybreak[0]\\
	&F^{+-}_{n,j} \df \left[\frac 2{3\qa} , \frac 5{6\qa}\right] + \frac j\qa = F^{-+}_{n,j}
	 + \frac 1{2\qa}
 \end{align*}
and for arbitrary $s_-, s_+\in\{+,-\}$
\[	
	F^{s_- s_+} \df \bigcap_{n=1}^\infty \bigcup_{j=0}^{\qa -1} F^{s_- s_+}_{n,j}.
\]
The sets $F^{s_- s_+}$ are nonempty and uncountable; indeed, every interval
$F^{s_- s_+}_{n-1,j}$ contains at least
 \begin{multline}
	\left\lfloor\frac{\lvert F^{s_- s_+}_{n-1,j} \rvert}{1/(\qa)}\right\rfloor - 1
	= \left\lfloor\frac \qa{6\qa[-1]}\right\rfloor - 1 \\= \left\lfloor\frac 16 \frac {\A\qk[-1]}
	{\A[-1]\qk} \cdot \frac{\qk}{\qk[-1]} \cdot \frac{\q}{\q[-1]}\right\rfloor
	- 1 \stackrel{\eqref{e5},\, \eqref{e1}}> \lfloor \frac 16 \cdot \frac{18}{25} \cdot 25\rfloor - 1 = 2
	\label{e18}
 \end{multline}
of the intervals $F^{s_- s_+}_{n,j}$, since the intervals from the $n$-th union 
are uniformly distributed with period $1/(\qa)$; therefore, the intersections
$F^{s_- s_+}$ are topological Cantor sets.

We will now show that the products $F^{s_-s_+}\times\R$ consist of \emph{discrete points},
i.e.\ points with discrete orbits, which proves that
$T_\f$ is a Besicovitch cylinder.
	More precisely, we will show that for every $x\in F^{s_- s_+}$:
\begin{itemize}
	\item if $s_+ = s_-$, then
\[
	\f^{(m)}(x)\xrightarrow{m\to \pm\infty} s_+\,\infty,
\]
	\item if $s_+ \neq s_-$, then
\[
	\f^{(m)}(x)\xrightarrow{m\to \pm\infty} (-1)^{k_n} s_{\pm}\,\infty,
\]
where the coefficient $(-1)^{k_n}$ is constant (cf.\ \eqref{e17}).
\end{itemize}
Later, in the next section, we will verify that these sets are of full Hausdorff dimension.

\subsection{The case of \texorpdfstring{$F^{++}$}{F++}}
Fix an element $x\in F^{++}$ and an integer $|m| > \qqk[1]/(3A_1)$. We wish to bound the
summands $f_l(x+m\alpha) -f_l(x)$ from below. To this end, recall that $x$ determines a sequence
$(j_l)^\infty_{l=1}$ such that $x\in F^{++}_{l,j_l}$ for every $l\in\N$ and let $x_l$ be
given by $x_l \df x - j_l/(\qqa)$; then $|x_l| \leq 1/(12\qqa)$. Now, by the
properties of $f_l$
 \begin{multline}
	f_l(x+m\alpha) - f_l(x) = f_l(x_l+m\alpha) - f_l(x_l) = f_l(x_l+m\alpha) \\
	= f_l\left(x_l+m\alpha -	\frac{m A_l p_{k_l}}{\qqa}\right) = f_l\left(x_l+m\left(\alpha -
	\frac{p_{k_l}}{\qq}\right)\right),
	\label{e7}
 \end{multline}
which implies that
 \begin{gather}
	\label{e10}
	f_l(x+m\alpha) - f_l(x)\geq 0.
 \end{gather}
Because of \eqref{e11}, there exists a unique $n = n(m)$ which satisfies
\[
	\qk[-1]/(2\A[-1])\leq |m|< \qk/(2\A),
\]
and when $|m|$ tends to infinity, so does $n(m)$. Sucvh assumption enables us to estimate the $n$-th summand
of $\f^{(m)}$:
\begin{align*}
	&\bullet\ \left\lvert m\left(\alpha - \frac{p_{k_n}}{\q}\right) \right\rvert \stackrel{\eqref{e0}}< \frac \qk{2\A}\cdot
	 \frac 1{\q\qk} = \frac 1{2\qa},\\
	&\bullet\ \left\lvert m\left(\alpha - \frac{p_{k_n}}{\q}\right) \right\rvert \stackrel{\eqref{e0}}> \frac {\qk[-1]}{2\A[-1]}
	 \cdot \frac 1{2\q\qk} = \frac 1{4\qa} \cdot \frac {\qk[-1]}{\A[-1]} \cdot \frac\A\qk\\
	 &\phantom{\bullet\ m\left(\alpha - \frac{p_{k_n}}{\q}\right)> \frac {\qk[-1]}{2\A[-1]}
	 \cdot \frac 1{2\q\qk} =\quad }
	 \stackrel{\eqref{e5}}> \frac 1{4\qa} \cdot \frac{18}{25} = \frac 9{50\qa}.
\end{align*}
Therefore, owing to the bound for $x_n$,
\[
	\cramped%
	{\frac 1{12} + \frac 1{75} = \frac 9{50} - \frac 1{12} < \left\lvert\!\left(x_n + m\alpha - \frac{mp_{k_n}}{\q}\right) \! \qa \!\right\rvert < \frac 1{12} } + \frac 12 < 1 - \frac 1{12} - \frac 1{75}.
\]
This leads to the bound we seek, since $f_n$ is even, symmetrical and unimodal on $[0, 1/(\qa)]$:
 \begin{multline*}
	f_n\left( x_n+m\left(\alpha - \frac{p_{k_l}}{\qq}\right)\! \right) = f_n\left(\left\lvert x_n+m\left(\alpha - \frac{p_{k_l}}{\qq}\right)\!\right\rvert \right) \\
	> f_n\left(\frac 1{12\qa} + \frac 1{75\qa}\right)
	= L_n \cdot \frac 1{75\qa} = \frac \qk{75\A n^2} \xrightarrow{\eqref{e11}} \infty,
 \end{multline*}
and finally proves the required divergence:
 \begin{multline*}
	\f^{(m)}(x) \stackrel{\eqref{e8}} = \sum_{l=1}^\infty (f_l(x+m\alpha) - f_l(x))
	\stackrel{\eqref{e10}}\geq f_{n(m)}(x+m\alpha) - f_{n(m)}(x)\\
	\stackrel{\eqref{e7}}= f_n\left(x_n+m\left(\alpha - \frac{p_{k_n}}{\q}\right)\!\right)
	\xrightarrow{|m|\to\infty} \infty.
 \end{multline*}
\subsection{The case of \texorpdfstring{$F^{--}$}{F--}}
The behaviour of functions $f_n$ on the set $F^{--}_{n,j}$
is symmetrical to the situation on $F^{++}_{n,j}$, and the calculations are
analogous.
%\begin{align*}
%	\bullet\ m\left(\alpha - \frac{p_{k_n}}{\q}\right) &\stackrel{\eqref{e0}}< \frac \qk{2\A}\cdot
%	 \frac 1{\q\qk} = \frac 1{2\qa},\\
%	\bullet\ m\left(\alpha - \frac{p_{k_n}}{\q}\right) &\stackrel{\eqref{e0}}> \frac {\qk[-1]}{2\A[-1]}
%	 \cdot \frac 1{2\q\qk} = \frac 1{4\qa} \cdot \frac {\qk[-1]}{\A[-1]} \cdot \frac\A\qk\\
%	 &> \frac 1{4\qa} \cdot \frac{18}{25} = \frac{0.18}\qa 
%\end{align*}
%\begin{itemize}
%	\item $m\left(\alpha - \frac{p_{k_n}}{\q}\right) \stackrel{\eqref{e0}}< \frac \qk{2\A}\cdot
%	 \frac 1{\q\qk} = \frac 1{2\qa}$,
%	\item $m\left(\alpha - \frac{p_{k_n}}{\q}\right) \stackrel{\eqref{e0}}> \frac {\qk[-1]}{2\A[-1]}
%	 \cdot \frac 1{2\q\qk} = \frac 1{4\qa} \cdot \frac {\qk[-1]}{\A[-1]} \cdot \frac\A\qk >
%	 \frac 1{4\qa} \cdot \frac{18}{25} = \frac{0.18}\qa \frac 9{50\qa}$
%\end{itemize}
\subsection{The case of \texorpdfstring{$F^{-+}$}{F-+} and \texorpdfstring{$F^{+-}$}{F+-}}
Choose an $x\in F^{-+} \cup F^{+-}$. Again, there is
$(j_l)^\infty_{l=1}$ such that $x\in F^{-+}_{l,j_l} \cup F^{+-}_{l,j_l}$ for every
$l\in\N$, and we denote by $x_l$ the respective ``reductions'' $x - j_l/(\qqa)$; then 
\[
	x_l \in \left[\frac 1{6\qqa}, \frac 1{3\qqa}\right]\ (s_+=+) \text{\quad or\quad}
	x_l\in \left[\frac 2{3\qqa}, \frac 5{6\qqa}\right]\ (s_+=-).
\]
%To avoid confusion, we will assume for a while that $s=+$; the second case requires
%only the obvious changes.
Additionally, fix an integer $|m| > \qqk[1]/(12 A_1)$. It follows
from the periodicity of $f_l$ that
 \begin{multline}
	f_l(x+m\alpha) - f_l(x) = f_l(x_l + m\alpha) - f_l(x_l)\\
	= f_l(x_l + m(\alpha - p_{k_l}/\qq)) - f_l (x_l).
	\label{e9}
 \end{multline}
We remind that $\operatorname{sign} (\alpha - p_{k_l}/\qq) \stackrel{\eqref{e0}}= (-1)^{k_l}
\stackrel{\eqref{e17}}= (-1)^{k_1}$. Take now $n=n(m)\geq 1$ for which
\begin{equation}
	\qk/(12\A)\leq |m|< \qk[+1]/(12\A[+1]). \label{e15}
\end{equation}
These constraints along with the inequalities \eqref{e0} imply that for $l>n$
 \begin{align*}
	\left\lvert m\left(\alpha - \frac{p_{k_l}}{\qq}\right)\right\rvert
	\stackrel{\eqref{e0}}< \frac {\qk[+1]}{12\A[+1]} \cdot \frac 1{\qq\qqk}
	\stackrel{\eqref{e6}}\leq \frac {\qqk}{12A_l} \cdot \frac 1{\qq\qqk} = \frac 1{12\qqa}.
 \end{align*}
Therefore, both arguments $x_l + m(\alpha - p_{k_l}/\qq)$
and $x_l$ lie in the same interval of linearity (and monotonicity) of
$f_l$, so the sign of the difference \eqref{e9} equals $(-1)^{k_1} s_+\operatorname{sign} m$
(it does not depend on $l$) and the expression \eqref{e9} can be estimated:
 \begin{multline*}
	\left\lvert f_l\left(x_l + m\left(\alpha - \frac{p_{k_l}}{\qq}\right)\right) - f_l(x_l) \right
	\rvert = L_l \left\lvert m\left(\alpha - \frac{p_{k_l}}{\qq}\right)\right\rvert\\
	\stackrel{\eqref{e0},\, \eqref{e15}}> \frac {\qq\qqk}{l^2} \cdot \frac {\qk}{12\A} \cdot
	\frac 1{2\qq\qqk} = \frac {\qk}{24\A l^2}.
 \end{multline*}
Since all these differences are of the same sign, this yields an estimate for the part of the
sum \eqref{e8} with $l>n$:
 \begin{equation}
	\cramped{%
	\left\lvert \sum_{l>n} \left(f_l\!\left(x_l + m\!\left(\alpha - \frac{p_{k_l}}{\qq}\right)\!
	\right) - f_l(x_l)\right) \right\rvert > \frac {\qk}{24\A} \sum_{l>n}
	\frac1{l^2} \stackrel{(\star)}> \frac {\qk}{25\A n},%
	}
	\label{e12}
 \end{equation}
where the inequality $(\star)$ holds for $n$ large enough, which results form the fact that
the remainder $\sum_{l>n} 1/l^2$ is asymptotically equivalent to $1/n$ (thus greater than
$24/(25n)$ for large $n$).%
	\footnote{This follows from the termwise equivalence to a telescoping series of
	$1/n$:\\
	\phantom{.}\hfill\(
	\sum_{l\geq n+1} \frac 1{(l+1)^2} - \sum_{l\geq n+2} \frac 1{(l+1)^2} = 1/(n+1)^2
	\approx (1/n) - 1/(n+1),
	\)\hfill\hfill \\
	and from an analogue of the Stolz-Ces\`aro Theorem.}

As it occurs, we do not have to work hard to take the remaining summand into account --
it suffices to subtract the upper bounds of the functions $f_l$:
 \begin{equation}
	\left\lvert \sum_{l\leq n} (f_l(x_l + m\alpha) - f_l(x_l)) \right\rvert
	\leq \sum_{l\leq n} 2\max_{x\in\T} f_l = \frac 23\sum_{l\leq n} \frac {\qqk}{A_l l^2}
	\label{e14}
 \end{equation}
Note that this sum behaves roughly like the sum of a finite geometric series: since $\qqk/A_l$
grows exponentially and, asymptotically, $l^2$ grows slower, the quotient for large $l$ also grows exponentially, say:
\[
	\frac {\qqk}{A_l l^2} \geq C \frac {\qqk[l-1]}{A_{l-1} (l-1)^2} \text{\quad for some } C>1
	\text{ and } l \text{ large enough}
\]
(e.g.\ when $l^2/(l-1)^2 < 1.1/C$). Then, indeed, the sum \eqref{e14} is of order of its largest
term, and therefore we arrive at a satisfactory bound:
%Then
%\[
%	\frac {\qk[-1]}{\A[-1] (n-1)^2} \geq C \frac {\qk[-2]}{\A[-2] (n-2)^2} \geq \cdots \geq
%	C^{n-1-l},
%\]
%and t
 \begin{equation}
	\frac 23 \sum_{l\leq n}\frac{\qqk}{A_l l^2} \leq\frac 23 \sum_{l\leq n}\frac 1{C^{n-l}}
	\cdot\frac{\qk}{\A n^2}	< \frac 23 \cdot \frac C{C-1}\cdot \frac{\qk}{\A n^2}
	\stackrel{(\star\star)}< \frac{\qk}{50 \A n},
	\label{e13}
 \end{equation}
where the inequality $(\star\star)$ also holds for large $n$. Combining the estimations
\eqref{e12}, \eqref{e14} and \eqref{e13}, we eventually obtain the required divergence:
 \begin{multline*}
	|\f^{(m)}(x)| \stackrel{\eqref{e8},\, \eqref{e9}}{=\!=} \left\lvert \sum_{l\geq 1}
	\left(f_l\left(x_l + m\left(\alpha - {p_{k_l}}/{\qq}\right)\!\right) - f_l(x_l)\right)
	\right\rvert\\
	\geq\left\lvert \sum_{l>n(m)} \cdots \right\rvert - \left\lvert \sum_{l\leq n(m)} \cdots
	\right\rvert \stackrel{\text{\scriptsize (\ref{e12},\ref{e14},\ref{e13})}}>
	\frac{\qk}{25 \A n} - \frac{\qk}{50 \A n} = \frac{\qk}{50 \A n}
	\mathop{\xrightarrow{m\to\pm\infty}}\limits_{\eqref{e11}} \infty.
 \end{multline*}%\mathop{>}\limits^{(\ref{e12},\ref{e14},\ref{e13})}_{\eqref{e13}}
Also, the sign of $\f^{(m)}$ is correct, because the prevailing part has correct sign.

\begin{rem}
Observe that the calculations for $F^{+-}$ and $F^{-+}$ (in this and the previous
section) do not require all the assumptions on $k_n$ and \A\ that we have made initially.
Actually, we only need that $k_n$ are all of the same parity, $\qk/\A$ grows at least
geometrically, and
\(
	\qa \geq 18 \qa[-1].
\)
In particular, the restriction for the growth of $\qk/\A$ (as in \eqref{e5}) is redundant
-- for example, we may put $\A\df 1$ for every $n$ (then we have to ensure the inequality
$\q\geq 18\q[-1]$). Moreover,
we do not use the pieces of constant value of the functions $f_n$. Summarizing, the sets
$F^{+-}$ and $F^{-+}$ also consist of discrete points in the following example from
\cite[Section 2]{flem}:
 \begin{gather}
	\f(x) \df \sum_{n=1}^\infty (g_n(x+\alpha) - g_n(x))
	\nonumber\\
\intertext{where $g_n$ are \L-Lipschitz, $1/\q$-pe\-ri\-od\-ic continuous functions:}
 	g_n(x)\df
	 \begin{cases}
	\L x, &\text{for } 0\leq x \leq \dfrac 1{2\q},\\
	\L\left( \dfrac 1\q - x \right), &\text{for } \dfrac 1{2\q} \leq x \leq \dfrac 1\q.
	 \end{cases}
 \end{gather}
and $\q\geq 18\q[-1]$ (this coefficient can be decreased by widening $F^{s_-s_+}_{n,j}$
appropriately).
\end{rem}
\section{Hausdorff dimension of \texorpdfstring{$F^{s_- s_+}$}{F s-s+}}
To compute the Hausdorff dimension of $F^{s_- s_+}$, we will use methods from \cite{fal}
(Example 4.6 and Proposition 4.1):\\
{\itshape Consider a sequence of unions of a finite number of disjoint closed intervals
in $[\n]0,1\n()$ \emph{(here: the sequence $(\bigcup_{j=0}^{\qa -1} F^{s_- s_+}_{n,j})%
_{n\geq 1}$)}. Suppose that the intervals of the $n$-th union ($n\geq 1$) 
\begin{itemize}
	\item are of length at most $\delta_n$ and $\delta_n\to 0$,
	\item are separated by gaps of length at least $\e_n$ (with $\e_n> \e_{n+1}>0$),
	\item contain at least $m_{n+1}\geq 2$ and at most $\overline m_{n+1}$ intervals of the
	$(n+1)$-st union.
\end{itemize}
Then the Hausdorff dimension of the intersection of this sequence lies between the following
two numbers:
	\[
	\liminf _{n\to\infty} \frac {\log (m_2\cdots m_{n})}{-\log(m_{n+1}\e_{n+1})} \leq 
	\liminf _{n\to\infty} \frac {\log (\overline m_2\cdots \overline m_{n})}{-\log \delta_{n+1}}.
	\] 
}
First, note that $\delta_n = |F^{s_- s_+}_{n,j}| = 1/(6\qa)\to 0$. Next, observe that
 \[
	\e_n = \frac 1{\qa} - |F^{s_- s_+}_{n,j}| >  \frac 1{\qa} - \frac 1{6\qa}
	> \frac 1{2\qa}.
 \]
As for $m_n$ and $\overline m_n$, we have already checked that $m_n\geq 2$ (see \eqref{e18}),
but we need a more precise estimate. Using the inequality $\lfloor t \rfloor -1 > t/2$ for
$t\geq 3$, we conclude that:
\[
	m_n \geq \left\lfloor\frac {|F_{n-1,j}^{s_- s_+}|}{1/\qa}\right\rfloor - 1 = 
	\left\lfloor\frac \qa {6\qa[-1]}\right\rfloor - 1 \geq \frac \qa {12\qa[-1]}.
\]
On the other hand, only one more interval can fit into:
\[
	\overline m_n \leq \left\lfloor\frac {|F_{n-1,j}^{s_- s_+}|}{1/\qa}\right\rfloor
	\leq \frac \qa {6\qa[-1]}.
\]
Consequently:
 \begin{align*}
	m_2 \cdots m_n &\geq \frac {\qqa[2]}{12\qqa[1]} \cdots \frac \qa{12\qa[-1]}
	= \frac \qa{12^{n-1}\qqa[1]},\\
	m_{n+1}\e_{n+1} &\geq \frac 1{12} \frac{\qa[+1]}\qa \cdot \frac 1{2\qa[+1]} = \frac
	1{24\qa},\\
	\overline m_2 \cdots \overline m_n &\leq \frac \qa{6^{n-1}\qqa[1]},
 \end{align*}
% \]
hence eventually
 \begin{align*}
	\dim_H F^{s_-s_+} &\geq \liminf _{n\to\infty} \frac{\log \qa - (n-1)\log 12 - \log\qqa[1]}
	{\log\qa + \log 24}\\
	&= 1 - \limsup_{n\to\infty} \frac n{\log\qa}\cdot \log 12,\\
	\dim_H F^{s_-s_+} &\leq \liminf _{n\to\infty} \frac{\log \qa - (n-1)\log 6 - \log\qqa[1]}
	{\log\qa + \log 6} \\
	&= 1 - \limsup_{n\to\infty} \frac n{\log\qa}\cdot \log 6.
 \end{align*}
Let us remark that the coefficient $12$ can be lowered nearly to $6$, if $m_n$ are larger.
Nevertheless, under the assumption \eqref{e3} the dimension equals $1$.
\begin{con}
	\label{con}
	For every irrational rotation of \T\ there exists a Besicovitch cocycle
	such that the set $\mathcal{D}\times \R\ (\supset F^{s_-s_+}\times \R)$ of discrete
	points of the respective cylinder has Hausdorff dimension two.
\end{con}

\section{Discrete Devaney chaos}
The sole property of transitivity is enough for some dynamicists to call a dynamical system
chaotic. However, over the years multiple definitions for chaos have been proposed.
Let us recall the notion of the Devaney chaos, one of the most popular ones: a dynamical
system $(X,T)$ on a metric space $(X,d)$ is \emph{chaotic in the sense of Devaney} if:
 \begin{enumerate}
	\item it is transitive,
	\item the set of periodic points is dense,
	\item the system is \emph{sensitive}, i.e.\ there are points around every point $x\in X$
	(arbitrarily close) whose orbits at least once diverge far enough from the orbit of $x$:
	there is $\e>0$ such that for every $x\in X$ and $\delta>0$ there are $n>0$ and $y$ with
	$d(x,y)< \delta$ and $d(T^n(x), T^n(y))> \e$.
 \end{enumerate}
We remind that the the last condition follows from the remaining ones, if $X$ is infinite
(\cite{bbcd}, main theorem, or \cite{gw}, Corollary 1.4).

It occurs that the dynamical systems we consider in this article satisfy a bit more general
condition, namely,
with ``periodic orbits'' replaced by ``discrete orbits'' (note that both notions are
equivalent in compact spaces). We will call this property \emph{discrete Devaney chaos} and
check this fact in a moment. A similar generalization was proposed in \cite{gw}, with
``almost periodic'' (that is, contained in a minimal set) instead of ``periodic'' and it was
 shown that this, combined with transitivity, implies sensitivity, if $X$ is compact.
Note also, that there are no periodic points in cylinders over minimal rotations, so they
cannot be Devaney chaotic.

Recall first that a space or a set is \emph{boundedly compact} if bounded closed subsets are
always compact.%
	\footnote{Such spaces are also given other names in the literature: they are called
	 \emph{proper}, \emph{finitely compact}, \emph{totally complete}, \emph{Heine-Borel} or
	 having the \emph{Heine-Borel property} (not to be confused with the Heine-Borel [covering]
	 property, or precompactness, that is, ``every open cover has a finite subcover'').}
In particular, closed subsets of Euclidean spaces are boundedly compact. All such spaces are
complete and separable. Also, a system is called \emph{maximally sensitive}, if it sensitive with
every~\mbox{$\e<\operatorname{diam}(X)/2$}, and \emph{maximally chaotic}, if it is Devaney chaotic
and maximally sensitive (definitions introduced in \cite{alp}).
 \begin{thm}
	Let $X$ be an infinite, boundedly compact space without isolated points, and let $T$ be
	transitive with dense set of discrete points. Then the system is sensitive. If, moreover,
	the set of discrete nonperiodic points is dense, then the system is maximally sensitive.
 \end{thm}
\begin{proof}
Since $X$ is complete, separable and without isolated points, the system is even
positively transitive (it has a dense semi-orbit -- see \cite{ox}, p.\ 70; the proof was
recalled in \cite{dy1}, Proposition 2.1).
The set of discrete points consists of periodic points and nonperiodic discrete points,
both of which are invariant. Thus, one of these sets contains a positively transitive point in
its closure, and so it is dense. If periodic points are dense, then, by \cite{bbcd} or
\cite{gw}, the system is sensitive. The new result is when the second set is dense, what we
assume henceforth.

Any infinite (= nonperiodic) discrete orbit, by bounded compactness, has no bounded
subsequence, so $\operatorname{diam}(X) = \infty$. Fix then any $x\in X$, $\e>0$ and $\delta>0$.
In the $\delta$-neighbourhood of $x$ there is a point $y_1$ with dense semi-orbit and a discrete nonperiodic
point $y_2$. Then, for infinitely many $n>0$ the orbit of $y_1$ returns to $x$: $d(T^n(y_1),x)<\e$,
and on the other hand, for $n$ large enough the orbit of $y_2$ stays far away from $x$:
$d(T^n(y_2),x)>3\e$ (by bounded compactness again). Consequently, for some $n>0$: $d(T^n(y_1), T^n(y_2))>2\e$, and hence 
$d(T^n(x), T^n(y_1))>\e$ or $d(T^n(x), T^n(y_2))>\e$.
\end{proof}
\begin{rem}
The Besicovitch cylinders that we consider are of course transitive and have a dense set of
discrete points (we have found discrete points in $F^{s_-s_+}\times\R$, but their orbits are
dense). Therefore, there are examples of maximally discretely chaotic
systems with full-dimensional set of relatively ``regular'' (almost periodic, discrete)
points. This feature seems not to be studied so far. However, there are results about
full Hausdorff dimension of the set of points with nondense orbits, although they are rather
concerned with bounded orbits -- see e.g.\ \cite{kl, ur}.
\end{rem}


\begin{thebibliography}{QWERT}
\bibitem[AlPr]{alp} Alpern, S., Prasad, V. S.,
\textit{Maximally chaotic homeomorphisms of sigma-compact manifolds.}
Topology Appl. \textbf{105} (2000), no. 1, 103--112. \MR1761090 (2001j:37018)
\bibitem[BaBC]{bbcd} Banks, J., Brooks, J., Cairns, G., Davis, G., Stacey, P.,
\textit{On Devaney's definition of chaos.}
Amer. Math. Monthly \textbf{99} (1992), no. 4, 332--334.  \MR1157223 (93d:54059)
\bibitem[Be1]{bes1} Besicovitch, A. S., \textit{A problem on topological transformations
	of the plane}. Fund. Math. \textbf{28}, (1937). 61--65.
	\href{http://www.zentralblatt-math.org/zmath/en/advanced/?q=an:63.0566.02}%
	{Zentralblatt: 63.0566.02}
\bibitem[Be2]{bes2} Besicovitch, A. S., \textit{A problem on topological transformations
	of the plane. II.} Proc. Cambridge Philos. Soc. \textbf{47}, (1951). 38--45.
	\MR0039247 (12,519e)
\bibitem[CoFS]{foko} Cornfeld, I.\ P., Fomin, S.\ W., Sinai, J.\ G. \textit{Ergodic theory}.
[English translation from Russian]. Springer-Verlag, New York, 1982. \MRa0832433 (87f:28019),
	Russian	original: \MR0610981 (83a:28017).
\bibitem[Dy]{dy1} Dymek, E., \textit{Introduction to Besicovitch transformations}.
	Term paper within the SSDNM programme, Nicolaus Copernicus University in Torun,
	2011. Available online on the SSDNM webpage:
	\href{http://ssdnm.mimuw.edu.pl/pliki/prace-studentow/st/pliki/eugeniusz-dymek-1.pdf}
	{\mbox{ssdnm.}mimuw.edu.pl}.
%\bibitem[Dy2]{dy2} Dymek, E., \textit{Cylinders with full-dimensional set of discrete orbits}.
%	Term paper within the SSDNM programme, Nicolaus Copernicus University in Torun,
%	2011. Available online on the SSDNM webpage:
%	\href{http://ssdnm.mimuw.edu.pl/pliki/prace-studentow/st/pliki/eugeniusz-dymek-2.pdf}
%	{\mbox{ssdnm}.mimuw.edu.pl}.
\bibitem[Fa]{fal} Falconer, K., \textit{Fractal geometry. Mathematical foundations and
	applications}. John Wiley \& Sons, Ltd., Chichester, 1990. \MR1102677 (92j:28008)
\bibitem[FrLe]{flem} Fr\k aczek, K., Lema\'nczyk, M., \textit{On the Hausdorff dimension
	of the set of closed orbits for a cylindrical transformation}. Nonlinearity
	\textbf{23}  (2010),  no. 10, 2393--2422. \MR2672680 (2011k:37015),
	\href{http://arxiv.org/abs/1006.4498}{arXiv: 1006.4498}
\bibitem[GlWe]{gw} Glasner, E., Weiss, B.,
\textit{Sensitive dependence on initial conditions.}
Nonlinearity~\textbf{6} (1993), no. 6, 1067--1075.  \MR1251259 (94j:58109)
\bibitem[GoHe]{gohe} Gottschalk, W. H., Hedlund, G. A., \textit{Topological dynamics.}
	American Mathematical Society Colloquium Publications, Vol. 36. American
	Mathematical Society, Providence, R. I., 1955. \MR0074810 (17,650e)
\bibitem[Kh]{khi} Khinchin, A. Ya., \textit{Continued Fractions}, Chicago University
	Press, Chi\-ca\-go-Lon\-don, 1964. \MR0161833 (28 \#5037)
\bibitem[Kl]{kl} Kleinbock, D. Y., \textit{Nondense orbits of flows on homogeneous spaces.}
Ergodic Theory Dynam. Systems \textbf{18} (1998), no. 2, 373--396. \MR1619563 (99e:58122)
\bibitem[KwSi]{ksie} Kwiatkowski, J., Siemaszko, A.,
\textit{Discrete orbits in topologically transitive cylindrical transformations}.
Discrete Contin. Dyn. Syst. \textbf{27} (2010), no. 3, 945--961.
\bibitem[LeMe]{leme} Lema\'nczyk, M., Mentzen, M. K., \textit{Topological ergodicity of
	real cocycles over minimal rotations.} Monatsh. Math. \textbf{134} (2002), no. 3,
	227--246. \MR1883503 (2003a:37014)
\bibitem[MaSh]{mash} Matsumoto, S., Shishikura, M., \textit{Minimal sets of certain
	annular homeomorphisms.} Hiroshima Math. J. \textbf{32} (2002), no. 2, 207--215.
	\MRa1925898 (2003f:37071)
\bibitem[MeSi]{mesi} Mentzen, M. K., Siemaszko, A., \textit{Cylinder cocycle extensions
	of minimal rotations on monothetic\od groups}. Colloq.\od  Math.\od \textbf{101}\od
	(2004),\od  no.\od 1,\od 75--88. \MR2106183 (2005g:54070)
\bibitem[Ox]{ox}  Oxtoby, J. C., \textit{Measure and category. A survey of the analogies between topological and measure spaces}. Graduate Texts in Mathematics, Vol. 2. Springer-Verlag, New York-Berlin, 1971. \MR0393403 (52 \#14213)
    \bibitem[Po]{poi} Poincar\'e, H., \textit{Sur les courbes d\'efinies par les \'equations diff\'erentielles (IV)} (French) [On curves defined by differential equations]. J. math. pures appl. $4^\text e$ s\'erie, \textbf{2} (1886), 151--218.
	\href{http://http://www.zentralblatt-math.org/zmath/en/search/?q=an:18.0314.01}%
	{Zentralblatt: 18.0314.01} Available on \href{http://portail.mathdoc.fr/JMPA/afficher_notice.php?id=JMPA_1886_4_2_A5_0}{Gallica-Math portal}.
\bibitem[Ur]{ur} Urba\'nski, M., \textit{The Hausdorff dimension of the set of points with
nondense orbit under a hyperbolic dynamical system}, Nonlinearity \textbf{2} (1991), 385--397.
\MR1107012 (92k:58204)
\end{thebibliography}
\end{document}